\documentclass[a4paper]{article}
\usepackage{amsmath,amsthm,amssymb,mathrsfs}
\usepackage{a4wide}
\usepackage{eucal}
\usepackage{stackrel}

\newcommand{\ie}{\emph{i.e.}}

\newcommand{\cf}{\emph{cf.}}
\newcommand{\etal}{\emph{et al.}}
\newcommand{\Real}{\mathbb{R}}
\newcommand{\Com}{\mathbb{C}}
\newcommand{\sii}{L^2}
\newcommand{\Hilbert}{\mathcal{H}}
\newcommand{\dom}{\mathop{\mathrm{dom}}\nolimits}
\newcommand{\sgn}{\mathop{\mathrm{sgn}}\nolimits}

\renewcommand{\d}{\mathrm{d}}

\newtheorem{theorem}{Theorem}
\newtheorem{lemma}{Lemma}

\newtheorem*{theorem*}{Theorem}

\theoremstyle{definition}
\newtheorem{remark}{Remark}

\begin{document}
%
\title{\Large\textbf{%
From Lieb--Thirring inequalities to spectral enclosures 
for the damped wave equation}}%
\author{David Krej\v{c}i\v{r}{\'\i}k \ and \ Tereza Kurimaiov\'{a}}	
\date{\small 
\begin{quote}
\begin{center}
\emph{
Department of Mathematics, Faculty of Nuclear Sciences and 
Physical Engineering, \\ Czech Technical University in Prague, 
Trojanova 13, 12000 Prague 2, Czech Republic
\\
david.krejcirik@fjfi.cvut.cz \& kurimter@fjfi.cvut.cz
}
\end{center}
\end{quote}
12 August 2020}
\maketitle
\begin{abstract}
\noindent
Using a correspondence between the spectrum of
the damped wave equation and non-self-adjoint Schr\"odinger operators,
we derive various bounds on complex eigenvalues of the former.
In particular, we establish a sharp result that 
the one-dimensional damped wave operator is similar
to the undamped one 
provided that the $L^1$ norm of the (possibly complex-valued)
damping is less than~$2$.
It follows that these small dampings are spectrally undetectable.
%
%
\end{abstract}
%

\section{Introduction}
%
Consider a classical or a quantum system
described by the \emph{damped wave equation}
\begin{equation}\label{dampedEq}
  u_{tt} + a(x) \, u_{t}-\Delta_x u=0
\end{equation}
in the space-time variables  $(x,t) \in \Real^d \times (0,\infty)$,
where the `damping'~$a$ is a complex-valued function. 
The positive (respectively, negative) part of the real part of~$a$
corresponds to the dissipation (respectively, excitation) 
of mechanical or electromagnetic waves, 
while the imaginary part of~$a$ can be interpreted 
as a conservative perturbation of a Dirac (quasi-)particle.

Writing $U := (u,u_t)^T$, it is customary to replace
the scalar second-order equation~\eqref{dampedEq}
by the first-order evolution system $U_t=A_a U$ 
with the matrix-valued \emph{damped wave operator}
\begin{equation}\label{dampedOp}
  A_{a} := 
  \begin{pmatrix}
    0 & 1 \\
    \Delta & -a
  \end{pmatrix}
  \,, \qquad
  \dom A_a := H^2(\Real^d) \times \dot{H}^1(\Real^d)
  \,, 
\end{equation}
acting in the Hilbert space 
$\Hilbert := \dot{H}^1(\Real^d) \times \sii(\Real^d)$.
A careful analysis of the stationary problem 
$A_a\Psi=\mu\Psi$ in~$\Hilbert$,
where $\mu \in \Com$ is a spectral parameter,
provides information on the behaviour of 
the time-dependent solutions~$u$ of~\eqref{dampedEq}. 
It is easy to see that the spectral problem is related 
to the operator pencil 
(see Lemma~\ref{lemma} below)
\begin{equation}\label{pencil}
  S_{\mu a}\psi :=
  -\Delta\psi+\mu \;\! a \;\! \psi = -\mu^2\psi 
  \qquad\mbox{in}\qquad
  \sii(\Real^d) \,.
\end{equation}

As an example of usefulness of the spectral data, 
let us recall the collaboration of the first author
with P.~Freitas~\cite{fk}, where it is shown that 
the damped wave equation~\eqref{dampedEq}
becomes unstable whenever the real-valued damping~$a$
achieves a sufficiently negative minimum.
The strategy of~\cite{fk} is based on establishing  
spectral \emph{asymptotics} of the family 
of Schr\"odinger operators~$S_{\mu a}$ as $\mu \to +\infty$. 
Consequently, $A_a$~possesses a real positive point~$\mu$ in the spectrum,
which is responsible for a global instability of~\eqref{dampedEq}. 
 
The present paper is partially motivated 
by a remark of T.~Weidl~\cite{Weidl-private}
that Lieb--Thirring inequalities known for~$S_{\mu a}$
could provide more quantitative information on
the location of the \emph{real} spectrum of~$A_a$.
Moreover, in view of the unprecedented interest
in non-self-adjoint Schr\"odinger operators in the near past,
new complex extensions of the Lieb--Thirring inequalities
have been derived in recent years. 
Consequently, the T.~Weidl's observation has become 
interesting for the \emph{complex} spectrum of~$A_a$ as well.
In this paper we go beyond the usual setting by  
considering even \emph{complex-valued} dampings.

The literature on damped wave systems is rather extensive
and we restrict ourselves on quoting the following recent works
and refer the interested reader for further references therein.
The relationship between the damped wave and Dirac equations
is discussed on an abstract level in~\cite{Gesztesy-Holden_2011}
and related one-dimensional Lieb-Thirring-type inequalities
can be found in~\cite{Cuenin-Laptev-Tretter_2014,cs}. 
An extensive spectral analysis of the wave operator
with possibly unbounded damping is performed in~\cite{fst}.
Basic eigenvalue bounds for weakly damped wave systems 
can be deduced from~\cite{Freitas_1999}. 
The perturbation of eigenvalues of abstract damped wave operators
has been recently studied in the framework of Krein spaces in
\cite{Nakic-Veselic_2020}.
Resolvent estimates for an abstract dissipative operator
are derived in \cite{Bouclet-Royer_2014,r}.
 
The structure of this paper is as follows. 
In Section~\ref{Sec.pre}
we comment on basic properties of 
the damped wave operator~$A_a$
and state a crucial correspondence between 
its eigenvalue problem 
and the operator pencil~\eqref{pencil}.
Self-adjoint and non-self-adjoint Lieb--Thirring-type 
inequalities are applied in Sections~\ref{Sec.real}
and~\ref{Sec.complex}, respectively.
In particular, in Theorem~\ref{Thm.Davies}
we prove that the point spectrum of
the one-dimensional damped wave operator~$A_a$
is empty provided that $\|a\|_{L^1(\Real)} < 2$.
We interpret this result as that small dampings cannot
be detected by measuring eigenfrequencies of the wave system.
In Section~\ref{simSection} we strengthen this observation 
to a complete lack of spectral detectability
by showing that the non-self-adjoint operator~$A_a$ is actually
similar to the skew-adjoint undamped operator~$A_0$
under the same smallness condition.
 
\section{Preliminaries}\label{Sec.pre}
%
Although the damped wave equation can be made meaningful
for certain unbounded dampings (\cf~\cite{fst}),
our standing hypothesis is that the damping 
$a:\Real^d \to \Com$ is bounded, \ie,
\begin{equation}\label{Ass}
  a \in L^\infty(\mathbb{R}^d)
  \,. 
\end{equation}
Then it is easy to see that there exists a constant~$c$
such that $A_a+c$ is a generator
of a $C^0$-semigroup of contractions (\cf~\cite[App.~B]{fk}),
so~\eqref{dampedEq} is well posed.
Equivalently, $A_a$~is a quasi-m-accretive operator. 
In particular, the operator~$A_a$ is closed,
so its spectrum is well defined.

Notice that~$A_a$ is skew-adjoint (\ie, $iA_a$~is self-adjoint)
if, and only if, $a$~is purely imaginary.
In particular, the undamped operator~$A_0$ is skew-adjoint.
To have this symmetry result, it is important that
we consider the homogeneous Sobolev space $\dot{H}^1(\Real^d)$ 
(defined as the completion of $C_0^\infty(\Real^d)$
with respect to the norm $\psi \mapsto \|\nabla\psi\|$)
in the definition of the Hilbert space~$\Hilbert$,
see~\eqref{dampedOp}.
On the other hand, it is important to keep in mind
that $\dot{H}^1(\Real^d)$ is not a subset of $\sii(\Real^d)$.

With an abuse of notation, 
we denote by~$-\Delta$ the distributional Laplacian
as well as its self-adjoint realisation in $\sii(\Real^d)$
with domain $\dom(-\Delta) := H^2(\Real^d)$.
For any bounded potential $V:\Real^d \to \Com$,
the Schr\"odinger operator $S_V = -\Delta+V$ 
with $\dom S_V = H^2(\Real^d)$ is a well defined
m-sectorial operator.

We say that~$V$ is \emph{vanishing at infinity}
and write $V\xrightarrow[]{\infty}0$,  provided that
\begin{equation*}
  \lim_{R\to+\infty}\|V\|_{L^\infty(\mathbb{R}^d\setminus B_R(0))}=0
  \,, 
\end{equation*}
where $B_R(0)$ denotes the ball of radius $R>0$
centred at the origin.
Then~$V$ is a relatively compact perturbation of~$-\Delta$. 
Consequently,
$
  \sigma_\mathrm{e}(S_V) 
  = \sigma_\mathrm{e}(S_0) 
  = [0,+\infty)
$
for any choice of the essential spectrum 
(namely, $\sigma_\mathrm{e1},\dots,\sigma_\mathrm{e5}$ 
in the notation of \cite[Chapt.~IX]{Edmunds-Evans}).
More specifically, it is clear for the components
$\sigma_\mathrm{e1},\dots,\sigma_\mathrm{e4}$,
which are preserved by relatively compact perturbations,
and the result extends to the widest choice $\sigma_\mathrm{e5}$
because~$S_V$ is m-sectorial 
(\cf~\cite[Prop.~5.4.4]{KS-book}).
 
It is easy to see that the spectrum of 
the undamped operator~$A_0$ is purely continuous 
and equal to the imaginary axis; in particular,
$
  \sigma(A_0) = \sigma_\mathrm{e}(A_0) = i\Real 
$.
If~$a$ is vanishing at infinity, 
then the damping represents
a relatively compact perturbation of~$A_0$.
Consequently,
$
  \sigma_{\mathrm{e}k}(A_a) 
  = \sigma_{\mathrm{e}k}(A_0) = i\Real 
$
for $k=1,\dots,4$. 
To see that it is true also for~$\sigma_\mathrm{e5}$,
it is enough to notice that each connected component
of $\Com \setminus i\Real$ intersects the resolvent set of~$A_a$
due to~\eqref{Ass}. In summary,
\begin{equation*}
  a\xrightarrow[]{\infty}0
  \qquad \Longrightarrow \qquad
  \sigma_{\mathrm{e}}(A_a) 
  = \sigma_{\mathrm{e}}(A_0) = i\Real 
  \,.
\end{equation*}

Since the essential spectrum is independent of~$a$,
the present paper focuses on the point spectrum of~$A_a$.
The following lemma specifies the equivalence between
the spectral problem for~\eqref{dampedOp} 
and the operator pencil~\eqref{pencil}
in the case of eigenvalues (not necessarily discrete).    

\begin{lemma}\label{lemma}
Assume~\eqref{Ass}.
For every $\mu\in\mathbb{C}$,
\begin{equation*}
  -\mu^2\in\sigma_\mathrm{p}(S_{\mu a})
  \quad\Longleftrightarrow\quad
  \mu\in\sigma_\mathrm{p}(A_a)
  \,.
\end{equation*}
\end{lemma}
\begin{proof}
Assuming $-\mu^2\in\sigma_\mathrm{p}(S_{\mu a})$,
there exists a non-trivial function $\psi_1 \in H^2(\Real^d)$ 
such that $S_{\mu a}\psi_1=-\mu^2\psi_1$. 
Then 
$
  \Psi := (\psi_1,\mu\psi_1)^T \in \dom A_{a}
$ 
and
\begin{equation*}
  (A_{a}\Psi)^T
  =\left(\mu\psi_1,\Delta\psi_1- a \mu\psi_1\right)
  =\left(\mu\psi_1,-S_{\mu a}\psi_1\right)
  =\mu\left(\psi_1,\mu\psi_1\right)
  =\mu\Psi^T
  \,.
\end{equation*}
Conversely, assuming $\mu\in\sigma_\mathrm{p}(A_{a})$,
there exists a non-trivial 
$\Psi=(\psi_1,\psi_2)^T \in \dom A_{a}$ 
such that $A_{a}\Psi=\mu\Psi$. 
In other words, 
$\psi_1\in H^2(\mathbb{R}^d)$,
$\psi_2\in \dot{H}^1(\mathbb{R}^d)$ 
and $\psi_2=\mu\psi_1$, 
$\Delta\psi_1- a \psi_2=\mu\psi_2$. 
Combining the latter two equations, 
we get $S_{\mu a}\psi_1=-\mu^2\psi_1$
with $\psi_1 \not= 0$.
\end{proof}
%
 
\section{Real spectrum and real-valued damping}\label{Sec.real}
%
In this section, 
we assume that the damping~$a$ satisfying~\eqref{Ass}
is \emph{real-valued} and focus on 
\emph{real} eigenvalues $\mu\in\mathbb{R}$.
Recalling the correspondence of Lemma~\ref{lemma},
it is enough to consider the auxiliary 
Schr\"odinger operators~$S_V$
with real-valued potentials~$V$.
Then~$S_V$ is self-adjoint, so its spectrum is purely real.
Let us denote by $\{\lambda_n(V)\}_{n=1}^N$
the non-decreasing sequence of negative discrete eigenvalues of~$S_V$,
where each eigenvalue is repeated according 
to its multiplicity.
The set can be either empty ($N=0$) or finite ($1 \leq N < \infty$)
or infinite ($N=\infty$).

Our starting point are the self-adjoint Lieb--Thirring inequalities 
(\cf~\cite{lw,l})
	\begin{equation}\label{realLT}
		\sum_{n=1}^N |\lambda_n(V)|^\gamma
		\le L_{\gamma,d}
		\int_{\mathbb{R}^d}V_{-}^{\gamma+\frac{d}{2}}
		\,,
	\end{equation}
where~$L_{\gamma,d}$ is a positive constant independent of~$V$
and $V_{\pm}:=\frac{1}{2}(|V|\pm V)$.
More specifically, it is known that such a bound holds 
with a \emph{finite} constant~$L_{\gamma,d}$ if, and only if,
\begin{equation}\label{ranges}
\begin{aligned}
  &\gamma\ge\mbox{$\frac{1}{2}$} & \mbox{if} \quad d=1 \,,
  \\
  &\gamma>0 & \mbox{if} \quad d=2 \,,
  \\
  &\gamma\geq 0 & \mbox{if} \quad d \geq 3 \,.
\end{aligned}
\end{equation}
The sharp values of the constants $L_{\gamma,d}$ are not known explicitly 
for all the admissible values of~$\gamma$. 
In special cases, however, one knows that 
$L_{\frac{1}{2},1}=\frac{1}{2}$ and 
$$
  L_{\gamma,d}
  =L_{\gamma,d}^\mathrm{cl}
  :=\displaystyle
  \frac{\Gamma(\gamma+1)}{2^d\pi^{\frac{d}{2}}\Gamma(\gamma+\frac{d}{2}+1)}
  \qquad \mbox{for} \qquad
  d\ge 1,\:\gamma\ge\frac{3}{2}
$$
are the best possible values.
 
A direct combination of the Lieb--Thirring inequalities
with Lemma~\ref{lemma} yields the following basic results.		
	\begin{theorem}\label{theorem1}
Suppose that~$a$ is real-valued and assume~\eqref{Ass}.
Let~$\gamma$ be any number satisfying~\eqref{ranges}.
If~$\mu$ is a positive (respectively, negative) eigenvalue of~$A_a$
and $a_-\in L^{\gamma+\frac{d}{2}}(\mathbb{R}^d)$
(respectively, $a_+\in L^{\gamma+\frac{d}{2}}(\mathbb{R}^d)$), then
		\begin{equation*}
		\mu^{\gamma-\frac{d}{2}}
		\le L_{\gamma,d}
		\int_{\mathbb{R}^d}a_-^{\gamma+\frac{d}{2}} 
		\qquad
		\left(
		\mbox{respectively, }
		(-\mu)^{\gamma-\frac{d}{2}}
		\le L_{\gamma,d}
		\int_{\mathbb{R}^d}a_+^{\gamma+\frac{d}{2}} 
		\right)
		\,.
		\end{equation*}
		On the other hand, 
		if $a_-\in L^d(\mathbb{R}^d)$ 
		(respectively, $a_+\in L^d(\mathbb{R}^d)$)
		and
		\begin{equation*}
		\int_{\mathbb{R}^d}a_-^d<\frac{1}{L_{\frac{d}{2},d}}
		\qquad
		\left(
		\mbox{respectively, }
		\int_{\mathbb{R}^d}a_+^d<\frac{1}{L_{\frac{d}{2},d}} 
		\right)
		\,,
		\end{equation*}
		then~$A_a$ has no positive 
		(respectively, negative) eigenvalues.
	\end{theorem}
\begin{proof}
From Lemma \ref{lemma} we get that 
real $\mu\in\sigma_\mathrm{p}(A_a)$ if, and only if,
there exists $n\in\mathbb{N}$
such that $\lambda_n(\mu a)=-\mu^2$.
 	Hence, \eqref{realLT} implies 
 	\begin{equation}\label{lt-both}
\begin{cases} 	
\displaystyle
	 \mu^{2\gamma} =|\lambda_{n}(\mu a)|^\gamma\le L_{\gamma,d}\:\mu^{\gamma+\frac{d}{2}}\int_{\mathbb{R}^d}a_-^{\gamma+\frac{d}{2}}
	 &\mbox{for}\qquad
	 \mu\in\sigma_\mathrm{p}(A_a) \cap (0,+\infty) \,,
	 \\
\displaystyle
	 |\mu|^{2\gamma}=|\lambda_{n}(\mu a)|^\gamma\le L_{\gamma,d}\:|\mu|^{\gamma+\frac{d}{2}}\int_{\mathbb{R}^d}a_+^{\gamma+\frac{d}{2}}
	 &\mbox{for}\qquad
	 \mu\in\sigma_\mathrm{p}(A_a) \cap (-\infty,0)
	 \,.
\end{cases}	 
	 \end{equation}
This proves the first part of the theorem.
	 The case $0\in\sigma_\mathrm{p}(\mathcal{A}_a)$ 
	 cannot occur because the spectrum of~$S_0$
	 is purely continuous.
	  Dividing~\eqref{lt-both}  by $|\mu|^{\gamma+\frac{d}{2}}$, and eventually putting $\gamma=\frac{d}{2}$, which can be done for all $d\ge 1$, we conclude with the second part of the theorem.
\end{proof}

To continue, we restrict ourselves to $d=1$. 
In this case the Buslaev--Faddeev--Zakharov trace formulae 
(\cf~\cite{fz}) provides us with a lower bound for the sum of 
square roots of the eigenvalues of~$S_{\mu a}$, namely
 	\begin{equation}\label{BFZ}
 	\sum_{n=1}^N|\lambda_n(\mu a)|^{\frac{1}{2}}
 	\ge-\frac{\mu}{4}\int_\mathbb{R}a \,.
 	\end{equation}
Of course, the inequality is non-trivial only if 
$\mu \int_{\Real} a < 0$.
The latter is known to be a sufficient condition 
which guarantees that $\inf\sigma(S_{\mu a}) < 0$.
Assuming in addition $a\xrightarrow[]{\infty}0$,
it follows that~$S_{\mu a}$ possesses at least one negative eigenvalue.
The number of eigenvalues~$N$ of~$S_{\mu a}$
is controlled from above by
the Bargmann bound (\cf~\cite[Problem 22]{s4})
 	\begin{equation*}
 	N\le 1+|\mu|\int_\mathbb{R}|a(x)||x|\:\d x \,.
 	\end{equation*}
Consequently, for~$\mu$ satisfying the inequality 
	\begin{equation}\label{oneeig}
		|\mu|<\left(\int_\mathbb{R}|a(x)||x|\:\d x\right)^{-1}
	\end{equation} 
 	the operator $S_{\mu a}$ has exactly one negative eigenvalue $\lambda_1(\mu a)$ and we get from~\eqref{BFZ} the estimate
 	\begin{equation}\label{bfz}
 	|\mu|=|\lambda_{1}(\mu a)|^\frac{1}{2}
 	\ge-\frac{\mu}{4}\int_\mathbb{R}a
 	\,.
 	\end{equation}
 	This implies $\int_\mathbb{R}a\ge -4$ for $\mu>0$ and $\int_\mathbb{R}a\le 4$ for $\mu<0$. 
 	In summary, we have established the following result.
\begin{theorem}\label{theorembfz}
Let $d=1$.
Suppose that~$a$ is real-valued 
and in addition to~\eqref{Ass}
assume $a\in L^1(\mathbb{R},|x| \, \d x)$ 
and $a\xrightarrow[]{\infty}0$.
Let~$\mu$ be a real eigenvalue of~$A_a$. 
If $\mu>0$ and $\int_\mathbb{R} a<-4$ 
(or $\mu<0$ and $\int_\mathbb{R}a>4$), then
		\begin{equation*}
		|\mu|\ge\left(\int_\mathbb{R}|a(x)||x|\:\d x\right)^{-1}.
		\end{equation*}
	\end{theorem}

Finally, combining the Lieb--Thirring inequalities 
with the Buslaev--Faddeev--Zakharov trace formulae, 
we obtain the following quantitative bounds on the location of real eigenvalues of the one-dimensional damped wave operator.
The presence of the coupling parameter~$\alpha$ is useful
for studying the stability of solutions of~\eqref{dampedEq}
in the spirit of~\cite{Freitas_1996,Freitas-Zuazua_1996,fk}. 

\begin{theorem}\label{theorem3}
Let $d=1$.
Suppose that~$a$ is real-valued 
and in addition to~\eqref{Ass}
assume $a\in L^1(\mathbb{R},|x| \, \d x)$ 
and $a\xrightarrow[]{\infty}0$.
Let $\int_\mathbb{R}a < 0$ (respectively, $\int_\mathbb{R}a > 0$).
For any $\mu > 0$ (respectively, $\mu < 0$)
satisfying~\eqref{oneeig},
there exists exactly one $\alpha> 0$ satisfying
		\begin{equation*}
		2\left(\int_\mathbb{R}a_-\right)^{-1}
		\le\alpha\le - 4\left(\int_{\mathbb{R}}a\right)^{-1}
\qquad
\left(\mbox{respectively, }
2\left(\int_\mathbb{R}a_+\right)^{-1}
		\le\alpha\le 4\left(\int_{\mathbb{R}}a\right)^{-1}
\right)		
		\end{equation*}
		such that $\frac{\mu}{\alpha}$ 
		is an eigenvalue of $A_{\alpha a}$.
\end{theorem}
\begin{proof}
Assuming $\mu>0$ and $\int_\mathbb{R} a<0$ together with $a\xrightarrow[]{\infty}0$ and \eqref{oneeig}, 
the operator~$S_{\mu a}$ possesses exactly one negative eigenvalue. 
Applying \eqref{lt-both} with $\gamma=\frac{1}{2}$ and \eqref{bfz}, 
we get
 	\begin{equation*}
 	-\frac{\mu}{4}\int_\mathbb{R}a\le|\lambda_1(\mu a)|^{\frac{1}{2}}\le\frac{\mu}{2}\int_{\mathbb{R}}a_-.
 	\end{equation*}
 	Similarly, for $\mu<0$ and $\int_\mathbb{R} a>0$, 
 	we get
 	\begin{equation*}
 	-\frac{\mu}{4}\int_\mathbb{R}a\le|\lambda_1(\mu a)|^{\frac{1}{2}}\le\frac{|\mu|}{2}\int_{\mathbb{R}}a_+.
 	\end{equation*}
 	These estimates together with Lemma~\ref{lemma} prove the theorem.
\end{proof}
%

\section{Complex spectrum and complex-valued damping}\label{Sec.complex}
%
In this section, we use the recent progress in spectral theory
of non-self-adjoint Schr\"odinger operators 
and state results about \emph{complex} eigenvalues 
of the damped wave operator.
There is no obstacle to consider \emph{complex-valued} dampings
at the same time.	

Let us start with dimension $d=1$.
The celebrated result of E.~B.~Davies \etal\ 
(see \cite[Thm.~4]{aad} and \cite[Corol.~2.16]{Davies-Nath_2002})
states that every eigenvalue $\lambda(V)$ of~$S_V$ 
with $V \in L^1(\Real)$ satisfies the bound
\begin{equation}\label{af}
  |\lambda(V)|^\frac{1}{2}
  \le \frac{1}{2}\int_{\mathbb{R}}|V(x)| \, \d x
  \,.
\end{equation}
Moreover, the bound is known to be optimal for step-like 
potentials approximating the Dirac potential.
Now, assuming in addition to~\eqref{Ass} that $a \in L^1(\Real)$,
if~$\mu$ is any eigenvalue of~$A_a$
(necessarily it must be non-zero,
\cf~the proof of Theorem~\ref{theorem1}),
Lemma~\ref{lemma} ensures that there exists 
$\lambda(\mu a)\in\sigma_\mathrm{p}(S_{\mu a})$ such that
\begin{equation*}
		|\mu|=|\lambda_{j}(\mu a)|^\frac{1}{2}
		\le \frac{1}{2}\:|\mu|\int_{\mathbb{R}}|a(x)| \, \d x
		\,,
\end{equation*}
where the inequality is due to~\eqref{af}. 
Dividing by~$|\mu|$, it follows that the point spectrum
of~$A_a$ is empty provided that the $L^1$-norm of~$a$ is small,
namely $\|a\|_{L^1(\Real)} < 2$.
Let us summarise the observation into the following theorem.
\begin{theorem}\label{Thm.Davies}
Let $d=1$.
In addition to~\eqref{Ass} assume $a\in L^1(\mathbb{R})$. 
If $\|a\|_{L^1(\Real)}<2$, then $\sigma_\mathrm{p}(A_a)=\varnothing$. 
Moreover, the constant~$2$ is optimal.  
\end{theorem}
\begin{proof}
It remains to argue about the optimality. 
Our strategy is to show that for any number slightly greater than~$2$ 
there exists a damping~$W$ with the $L^1$-norm equal to this number such that~$A_W$ has an eigenvalue. For this we choose the analytically computable case, where the damping is a step-like potential
	\begin{equation*}
		W(x):=\begin{cases}
		0&\mbox{if}\quad x<-b \,, \\
		a&\mbox{if}\quad -b<x<b \,, \\
		0&\mbox{if}\quad x>b \,,
		\end{cases}
		\qquad\mbox{with}\qquad
		a<0<b \,.
	\end{equation*}
	The eigenvalue equation $A_W\Psi=\mu\Psi$ reduces to 
	$
		\Delta\psi-\mu W\psi-\mu^2\psi=0
	$,
	where $\psi$ is the first component of~$\Psi$. 
It is enough to analyse the situation of real~$\mu$.
For $\mu=0$ or $\mu=-a$ we get just a trivial solution to the eigenvalue equation. Also, using the spectral correspondence of Lemma~\ref{lemma}, 
we know that all the real eigenvalues of~$A_W$ must be positive, provided that $W\le 0$ and $\mu\le\|a\|_{L^\infty(\Real)}=-a$. Thus the only interval where we can find a real eigenvalue is $(0,-a)$. 
	
	Now, for $\psi$ to lie in $H^2(\mathbb{R})\subset C^1(\mathbb{R})$ and be non-trivial, it can be easily verified that the secular equation 
	\begin{equation*}
		F(\mu):=2\sqrt{-(\mu a+\mu^2)} \,
		\cos\big(2b\sqrt{-(\mu a+\mu^2)}\big)
		+(a+2\mu) \, 
		\sin\big(2b\sqrt{-(\mu a+\mu^2)}\big)=0
	\end{equation*}
	must be satisfied. For $\mu\in (0,-a)$ this is equivalent to 
$
  G(\mu):=\sqrt{-(\mu a+\mu^2)}F(\mu)=0
$.
	We compute 
	\begin{align*}
		\lim_{\mu\to 0^+}G(\mu)&=0=\lim_{\mu\to -a^-}G(\mu) \,, \\
		\lim_{\mu\to 0^+}G'(\mu)&=2a(-1+c) \,, \\
		\lim_{\mu\to -a^-}G'(\mu)&=2a(1+c) \,,
	\end{align*}
	where $c:=-ba=\frac{1}{2}\|W\|_{L^1(\Real)}$. 
	We observe that for $c>1$ both the derivatives have the same sign which, together with the fact that the limit points of this continuous function $G(\mu)$ are the same, implies that there exists $\mu^*\in(0,-a)$ such that $G(\mu^*)=0$. Hence, $\mu^*\in\sigma_\mathrm{p}(A_W)$ which proves the desired optimality. 
\end{proof}
\begin{remark}
Taking $b:=(2|a|)^{-1}$ and $\alpha \in \Com$,
the complexified step-like potential~$\alpha W$ 
converges in the sense of distributions 
to $\alpha \delta$ as $a \to - \infty$, 
where~$\delta$ is the Dirac delta function. 
Replacing (formally)~$a$ by~$\alpha\delta$ in~\eqref{pencil},
one arrives at the pencil problem
\begin{equation}\label{delta}
\left\{
\begin{aligned}
  -\psi'' &= -\mu^2\psi && \mbox{in}\quad \Real\setminus\{0\} \,,
  \\
  \psi(0^+) - \psi(0^-) &= 0 \,,
  \\ 
  \psi'(0^+) - \psi'(0^-) &= \mu\alpha \;\! \psi(0) \,.
\end{aligned}
\right.
\end{equation}
There exists no admissible solution $\psi \in H^2(\Real\setminus\{0\})$, 
unless $\alpha=-2$ (respectively, $\alpha=2$)
in which case every $\mu \in \Com$ with $\Re\mu > 0$
(respectively, with $\Re\mu < 0$) 
is an `eigenvalue'!
Since $\|\alpha W\|_{\sii(\Real)}=|\alpha|$,
this is another support for the optimality
of the constant~$2$ in Theorem~\ref{Thm.Davies}.
However, we are not aware of any result about 
a spectral approximation of the operator pencil~\eqref{pencil} 
by bounded potentials. 

The damped wave equation on a \emph{finite} interval
with the damping being a Dirac delta function 
was previously studied in~\cite{Bamberger-Rauch-Taylor_1982}, 
\cite[Sec.~4.1.1]{Ammari-Nicaise} and \cite{Cox-Henrot_2008}.
\end{remark}

In Section~\ref{simSection}, we argue that the absence of eigenvalues 
follows as a consequence of the similarity of~$A_a$ 
to the undamped wave operator~$A_0$, 
provided that $\|a\|_{L^1(\Real)}<2$. 
 
\medskip 
 
In higher dimensions $d \geq 2$,
we use the robust result of R.~Frank \etal\ 
(see~\cite[Thm.~1]{f} and \cite[Thm.~3.2]{fs}) stating that
every eigenvalue $\lambda(V)$ of~$S_V$ 
with $V \in L^{\gamma+\frac{d}{2}}(\Real^d)$,
$0<\gamma\le \frac{1}{2}$,
satisfies the bound
\begin{equation}\label{ff}
  |\lambda(V)|^\gamma
  \le D_{\gamma,d}\int_{\mathbb{R}^d}|V(x)|^{\gamma+\frac{d}{2}} \, \d x
  \,,
\end{equation}
where~$D_{\gamma,d}$ is a constant independent of~$V$.
Using Lemma~\ref{lemma}, it follows that any eigenvalue~$\mu$ 
of~$A_a$ satisfies
	\begin{equation}\label{ff-consequence}
		|\mu|^{2\gamma}=|\lambda_{j}(\mu a)|^\gamma\le D_{\gamma,d} \, |\mu|^{\gamma+\frac{d}{2}}\int_{\mathbb{R}^d}|a(x)|^{\gamma+\frac{d}{2}} \, \d x
		\,.
	\end{equation}
We therefore conclude with the following theorem.	 
\begin{theorem}\label{theorem4}
Let $d \geq 2$.
In addition to~\eqref{Ass} 
assume $a\in L^{\gamma+\frac{d}{2}}(\mathbb{R}^d)$
with $0<\gamma\le\frac{1}{2}$.
There exists a constant~$D_{\gamma,d}$ such that,
for any eigenvalue $\mu\in\sigma_\mathrm{p}(A_a)$, 
one has
		\begin{equation*}
			|\mu|^{\gamma-\frac{d}{2}}\le D_{\gamma,d}\int_{\mathbb{R}^d}|a|^{\gamma+\frac{d}{2}} \,.
		\end{equation*}
	\end{theorem}

The formal analogue of~\eqref{ff} for $\gamma=0$ and $d \geq 3$ states
that there exists a dimensional positive constant~$D_{0,d}$
such that if 
\begin{equation}\label{ff-absence}
  D_{0,d}\int_{\mathbb{R}^d}|V(x)|^{\frac{d}{2}} \, \d x
  < 1
  \,,
\end{equation}
then~$S_V$ has no eigenvalues.  
The case of discrete eigenvalues is due to R.~Frank \cite[Thm.~2]{f},
while possibly embedded eigenvalues were covered by \cite[Thm.~3.2]{fs}.
Independently, the total absence of eigenvalues 
under a weaker hypothesis for $d=3$
and alternative conditions in higher dimensions
was obtained by L.~Fanelli, 
D.~Krej\v{c}i\v{r}\'ik and L.~Vega in~\cite{fkv}
(see also~\cite{FKV2} and~\cite{CK2,CFK}).
In fact, the spectral stability of~$S_V$ for potentials~$V$
which are small in some sense goes back to the abstract result of T.~Kato's~\cite{k}
that we shall recall for other purposes in the following section.
Here we just mention that a straightforward combination of Lemma~\ref{lemma}
and~\eqref{ff-absence} implies that~\eqref{ff-consequence}
holds also for $\gamma=0$ and $d \geq 3$.
Consequently, we get the following formal analogue of Theorem~\ref{theorem4}.

\begin{theorem}\label{Thm.absence}
Let $d \geq 3$.
In addition to~\eqref{Ass} 
assume $a\in L^{\frac{d}{2}}(\mathbb{R}^d)$.
There exists a positive constant~$D_{0,d}$ such that,
for any eigenvalue $\mu\in\sigma_\mathrm{p}(A_a)$, 
one has
		\begin{equation*}
			|\mu|^{-	\frac{d}{2}}\le D_{0,d}\int_{\mathbb{R}^d}|a|^{\frac{d}{2}} \,.
		\end{equation*}
\end{theorem}
%

\section{Similarity}\label{simSection}
%
In this section, we come back to
the one-dimensional setting of Theorem~\ref{Thm.Davies}.
We find the result appealing because it implies 
that small dampings cannot be distinguished from the undamped system 
just by measuring eigenfrequencies.
In fact, the following result shows a much stronger result 
that small dampings are indeed \emph{spectrally undetectable}.
\begin{theorem}\label{theoremSim}
Let $d=1$.
In addition to~\eqref{Ass} assume $a\in L^1(\mathbb{R})$. 
If $\|a\|_{L^1(\Real)}<2$, then the operators~$A_a$ and~$A_0$ 
are similar to each other.
Moreover, the constant~$2$ is optimal.  
\end{theorem}

Here the similarity means that there exists an operator
$W \in \mathscr{B}(\Hilbert)$ such that $W^{-1} \in \mathscr{B}(\Hilbert)$
and $A_a=WA_0W^{-1}$.
In other words, $iA_a$~is \emph{quasi-self-adjoint}
(\cf~\cite{KS-book}), because $iA_0$~is self-adjoint.
Then the spectra of~$A_a$ and~$A_0$ coincide.
In particular, $A_a$~must have the same eigenvalues as~$A_0$. 
Since the spectrum of~$A_0$ is purely continuous,
Theorem~\ref{Thm.Davies} follows as a direct consequence 
of Theorem~\ref{theoremSim}.

\begin{proof}
The optimality of the constant~$2$ follows from the proof
of Theorem~\ref{Thm.Davies}. Indeed, for any number strictly 
larger than~$2$, there exists a damping~$a$ whose $L^1$-norm 
equals that number, while~$A_a$ possesses eigenvalues. 
This would violate the similarity.

To prove the similarity, we use the abstract result 
of T.~Kato \cite[Theorem 1.5]{k}.
Writing 
\begin{equation*}
	i A_a = iA_0
	+(-i\;\!\sgn\overline{a} \, B)^* B
	\qquad\mbox{with}\qquad
	B := \begin{pmatrix}0&0\\0&|a|^\frac{1}{2}\end{pmatrix}
	\,,
\end{equation*}
where~$\sgn$ is the complex sign 
(defined by $\sgn f := f/|f|$ if $f \not=0$
and $\sgn f := 0$ if $f =0$),
it is enough to show that the bounded operators~$B$ 
and $-i\;\!\sgn\overline{a} \, B$
are \emph{relatively smooth} with respect to~$iA_0$
(\cf~\cite[Def.~1.2]{k}) and  
$$
  \sup_{\xi \in \Com\setminus\Real}
  \|K_\xi\| < 1 
  \,, \qquad \mbox{where} \qquad
  K_\xi := B R(\xi,iA_0) B i \sgn a 
$$
is the Birman--Schwinger operator.
Here we use the notation $R(\xi,T) := (T-\xi)^{-1}$ 
for the resolvent of an operator~$T$ at point 
$\xi \in \Com \setminus \sigma(T)$. 
The relative smoothness is a rather complicated condition in general,
but it reduces to reasonable criteria 
when~$iA_0$ is self-adjoint (\cf~\cite[Thm.~5.1]{k}).
Using in addition the simple structure
of the intertwining operators~$B$ 
and $-i\;\!\sgn\overline{a} \, B$ in our case, 
everything reduces to verifying the unique condition
\begin{equation}\label{unique}
  \sup_{\xi \in \Com\setminus\Real}
  \|\tilde{K}_\xi\| < 1 
  \,, \qquad \mbox{where} \qquad
  \tilde{K}_\xi := B R(\xi,iA_0) B  
  \,.
\end{equation}

It can be easily verified that
\begin{equation*}
R(\xi,i\mathcal{A}_0)
=
R(\xi^2,-\Delta)
\begin{pmatrix}
\xi 
&i  
\\
i\Delta  
&\xi  
\end{pmatrix}
.
\end{equation*}
Consequently, 
$$
  \tilde{K}_\xi
  = 
  |a|^{\frac{1}{2}} \xi R(\xi^2,-\Delta) |a|^{\frac{1}{2}}
\begin{pmatrix}
0 & 0
\\
0 & 1
\end{pmatrix}
,
$$
and therefore
$$
  \|\tilde{K}_\xi\| 
  = |\xi| \, \left\|
  |a|^{\frac{1}{2}} R(\xi^2,-\Delta) |a|^{\frac{1}{2}}
  \right\|
  ,
$$
where~$\|\cdot\|$ denotes the operator norm both 
in~$\Hilbert$ and~$\sii(\Real)$
on the left- and right-hand side, respectively. 
The latter will be estimated by the Hilbert--Schmidt norm
$\|\cdot\|_{\mathrm{HS}}$. 
To this purpose, we recall the explicit formula 
for the integral kernel of $\mathcal{R}(z,-\Delta)$:
 \begin{equation*}
	 G_z(x,y)=\frac{e^{-\sqrt{-z}\,|x-y|}}{2\sqrt{-z}} \,,
 \end{equation*}
where $x,y \in \Real$,
$z \in \Com\setminus[0,+\infty)$
and the principal branch of the square root is used.
Using the elementary estimate 
$$
  |G_z(x,y)| \leq \frac{1}{2 \sqrt{|z|}}
  \,,
$$ 
we therefore get 
$$
\begin{aligned}
  \|\tilde{K}_\xi\|^2 
  &\leq |\xi|^2 \, \left\|
  |a|^{\frac{1}{2}} R(\xi^2,-\Delta) |a|^{\frac{1}{2}}
  \right\|_{\mathrm{HS}}^2
  \\
  &= |\xi|^2 \int_{\Real\times\Real}
  |a(x)| \, |G_{\xi^2}(x,y)|^2 \, |a(y)|^2 \, \d x \, \d y
  \\
  &\leq \frac{\|a\|_{L^1(\Real)}^2}{4}
  \,, 
\end{aligned}  
$$
where the last bound is independent of~$\xi$.
Recalling~\eqref{unique}, T.~Kato's similarity condition 
\cite[Theorem 1.5]{k} holds provided that 
$\|a\|_{L^1(\Real)} < 2$. 
\end{proof}

\subsection*{Acknowledgment}
The research was partially supported 
by the GACR grants No.~18-08835S and 20-17749X.

%
\bibliography{mybib}
\bibliographystyle{amsplain}

\end{document}